\documentclass[final,1p,times]{elsarticle}

\usepackage{amssymb}
\newcommand{\R}{{\Bbb R}}
\newcommand{\Z}{{\Bbb Z}}
\newcommand{\N}{{\Bbb N}}

\usepackage{amsmath}

\newtheorem{thm}{Theorem}
\newtheorem{lemma}[thm]{Lemma}

\newtheorem{example}[thm]{Example}
\newtheorem{definition}[thm]{Definition}

\newproof{proof}{Proof}

\begin{document}

\begin{frontmatter}



\title{Almost periodic evolution systems with impulse action at state-dependent moments }

\author[a,d]{Robert Hakl}
\author[b]{Manuel Pinto}
\author[c]{Viktor  Tkachenko}
\author[d]{Sergei Trofimchuk\footnote{Corresponding author.}}
\address[a]{Institute of Mathematics, AS CR, Prague, Czech Republic
\\ {\rm E-mail: hakl@ipm.cz}}
\address[b]{Facultad de Ciencias, Universidad
de Chile,  San\-tia\-go, Chile  
\\ {\rm E-mail: pintoj@uchile.cl}}
\address[c]{Institute of Mathematics,
National Academy of Sciences of Ukraine,  Tereshchenkivs'ka str.
3,  Kyiv, Ukraine
\\ {\rm E-mail: vitk@imath.kiev.ua}}
\address[d]{Instituto de Matem\'atica y Fisica, Universidad de Talca, Casilla 747,
Talca, Chile \\ {\rm E-mail: trofimch@inst-mat.utalca.cl}}
\bigskip

\begin{abstract}
\noindent 
We study the existence of almost periodic solutions
for semi-linear abstract parabolic evolution equations
with impulse action at state-dependent moments. In particular, we present conditions excluding the beating phenomenon in these systems. The main result is illustrated with an example of  impulsive diffusive logistic equation.
\end{abstract}
\begin{keyword} Wexler's almost periodic solution; evolution system; impulse action at variable times; beating phenomenon \\

\vspace{3mm}

\end{keyword}

\end{frontmatter}

\newpage

\section{Introduction} \label{intro}
\noindent The studies of almost periodic and almost automorphic solutions constitute a significant part 
of the theory of impulsive systems.  Already in their seminal work  of 1968,  "Teoria calitativa a sistemelor cu impulsuri", A. Halanay and D. Wexler elaborated a framework showed to be adequate  to approach the topic of almost periodicity in different contexts of the theory of discontinuous semi-dynamical systems.  One of cornerstones of this framework was the concept of an almost periodic measure  \cite{Sch},   in the posterior works usually reduced to a simpler subclass of Wexler's almost periodic measures, cf. \cite{TrSP}.     
The Halanay and Wexler's book \cite{HW}  triggered the interest  of various researchers in the field of differential equations, and the monographs  \cite{Pankov,SP,S,TrSP} present the main achievements  of the almost periodic impulsive theory obtained during the last decades of the past century. As it was shown in \cite{STr,Tr,TrSP}, in the case when the consecutive moments $t_j$ of impulse action are uniformly separated, i.e. $\inf_j\{t_{j+1}-t_j\}\geq \theta >0$, almost periodic impulsive equations generate equivalent   
continuous semi-dynamical systems. In consequence,  various principles (e.g. formulated by Favard, Levitan, Zhikov)  of the classical almost periodic theory are also valid for the impulsive case.  In the latter context, sometimes it is convenient to replace  Wexler's concept of a piecewise continuous almost periodic function with more simple  definition of Stepanov  almost periodic solution \cite{STr,Tr,TrSP}. The case when $\inf_j\{t_{j+1}-t_j\}=0$ is much more complicated and can produce  various surprising effects.  For instance,  an exponentially dichotomic almost periodic linear inhomogeneous system typically does not possess any almost periodic solution. Nevertheless, it still has a unique essentially automorphic (more precisely, Levitan $N-$almost periodic) bounded solution  \cite{STr}. 

The recent years have again witnessed a growth of interest in the theory of almost periodic impulsive systems (and their applications as well,
cf. \cite{PG1,PG2,S}). 
First, because of an interesting and promising connection between almost periodic dynamic equations on time scales and almost periodic 
impulsive systems \cite{LMP}. It seems that  the studies of general almost periodic time scales \cite{WA} can benefit from the general theory 
of almost periodic sets on the real line \cite{STsets,TrSP}.  Second, we would like to mention a series of recent studies of topological impulsive 
semiflows, where the concept of an almost periodic motion plays one of the central roles. See \cite{BJ1,BJ2,BJ3} for more details and references. 
And, finally, the subclass of abstract impulsive systems seems now to be attracting much more interest from the experts in the field, e.g. see  
\cite{ABET,BJ2,HAR, EH1,EH2,PR,Sla,SA,OO,Tk2}.  Periodic and almost periodic solutions of these systems were investigated by many authors, 
we refer the reader to  \cite{HAR,Pankov,RogovY,Ronto,SA,Tk2,Tr} for some relevant results and  further references. Due to the complex nature
of abstract almost periodic impulsive equations,  they always were  considered with  pre-fixed (i.e. state-independent) moments of impulse action. 
However, these moments may depend on 
the current state of the evolutionary process \cite{AP,RT,Ronto,SP} that requires the analysis of almost periodic evolution systems 
with impulse action at variable times.  
In the present work, we are doing the first step in this direction, by 
investigating the abstract almost periodic system    
\begin{eqnarray} \label{ban1}
& & \frac{dx}{dt} + (A + A_1(t))x = f(t, x), \quad t \not= \tau_j(x), \\[2mm]
& & x(t + 0) - x(t) = g_j(x(t)), \quad  t = \tau_j(x(t)),\  j \in \mathbb Z, \label{ban2}
\end{eqnarray}
having impulsive forces located on the surfaces $ \Gamma_j: =\{(t,x): t= \tau_j(x)\}$ which are  uniformly separated 
each from other.  Here $x(t), \ t \in \R,$ belongs to a Banach space $X$, $A$ is a sectorial operator in $X$, and closed operators $A_1(t), \ t \in \R,$ are generally unbounded in $X$. Our main goal is to develop a new approach 
to the main challenges appearing in the studies of system (\ref{ban1}), (\ref{ban2}): (a) the fluctuation of points of discontinuities from one solution of  (\ref{ban1}), (\ref{ban2}) to another; (b) the beating phenomenon, when a  trajectory of  (\ref{ban1}), (\ref{ban2}) 
may hit the same surface $\Gamma_j$ several times; (c) an adequate election of functional spaces, in order to 
obtain `sufficiently strong' almost periodic solutions.  It should be observed here that in many works the possibility of beatings of trajectories is usually excluded at the very beginning of studies. It can be reached by assuming rather strong restrictions on the pairs $\{\tau_j,g_j\}$,   cf. \cite{BJ1}.  To simplify our exposition, we will also exclude the beatings; nevertheless, our restrictions seem to be rather moderate from the geometrical point of view, see Lemma 
\ref{L3v} below.  In any case, it is completely natural (at least, from the perspective of applications in mechanics) to ask about the existence of almost periodic regimes with beatings \cite{TrSP}. 
  Following the tradition, we will also invoke the usual definition of Wexler's  piecewise continuous almost periodic function. This requires auxiliary results similar to Lemma \ref{wep}  (a  `prototype' version of which can be found in \cite{HW}: here we include the proof of this lemma  for the completeness).  The main result of this paper is Theorem \ref{Thm7} in the third section.  It states the existence of 
an almost periodic solution to system (\ref{ban1}), (\ref{ban2}) under the assumptions of sufficient roughness 
of its linear part and smallness of the Lipschitz constants for all nonlinearities  $f(t,x), \tau_j(x),g_j(x)$ in the variable $x$. The aforementioned roughness is expressed in terms of the exponential dichotomies of evolution systems \cite{H,VIT}, and some properties of the  associated Green functions are analysed in Section \ref{SGF}. 

To prove the existence of almost periodic solution, we follow the strategy proposed  in \cite{RT} for the abstract periodic impulsive systems (\ref{ban1}), (\ref{ban2}) (and which differs from the reduction method proposed in \cite{AP} for the finite-dimensional systems). Namely, we construct some Poncar\'e type map $S$ in a set $\frak{N}$ of almost periodic sequences with values in the Banach space $X$.   Theorems \ref{T7} and \ref{Thm7} where this map is constructed and analysed are the most technically involved parts of this article. 
A unique fixed point of the map $S$ in $\frak{N}$ corresponds to an almost periodic solution for the impulsive system (\ref{ban1}), (\ref{ban2}). The main result is illustrated with an example of  impulsive diffusive logistic equation. 

\section{Basic definitions and preliminary results} 
\subsection{Almost periodicity} \noindent In the paper, $X$ denotes a Banach space provided with the norm $|\cdot|$.   A sequence of elements $x_k, \  k \in \Z,$ of $X$ is called almost periodic if  {\it  for every $\epsilon >0$ there exists $l >0$ such that each subinterval of \ $\R$ of the length $l$ contains some integer $p$ with the property }
$$
|x_{k+p} - x_k| < \epsilon, \quad k \in \Z. 
$$ 
The integer $p$ is called an $\epsilon$-almost period of the sequence $\{x_k\}$ and the above words in italic are usually 
shortened to "{\it the set of $\epsilon$-almost periods of $\{x_k\}$  is relatively dense}".  It is well known that 
each almost periodic sequence is bounded. 

Assume now that $a>0$,  $\{c_k\}$ is an almost periodic sequence of real numbers, and 
$\tau_k = a k + c_k, \ k \in \Z$, defines a strictly increasing sequence (hence, $\{\tau_k\}$ does not have finite limit points).  Following \cite{SP,STr}, we call such a collection of points $\{\tau_k\}$ strongly almost periodic subset of $\R$.  
We also will say that a bounded function $\phi:\R \to X$ is W-almost periodic [i.e. Wexler-almost periodic] with eventual discontinuities at $\tau_j$, if $\phi$ is uniformly continuous on the disjoint union of intervals $(\tau_{j-1}, \tau_j)$ and if for every $\epsilon > 0$ there exists  a relatively dense set $\Gamma$ such that  $$|\varphi(t +\tau) - \varphi(t)| < \epsilon \quad \mbox{once}\quad \tau \in \Gamma,\  |t - \tau_k| \ge \epsilon, \ k \in \mathbb{Z}.$$
Clearly, this definition mimics the classical concept of Bohr almost periodic function (i.e. 
continuous function $\phi : \mathbb{R} \to X$ possessing, for any $\epsilon > 0$, a relatively dense set $\Gamma$ of $\epsilon$-almost periods). 
In fact,  as it was shown in \cite[Lemma 5]{Tr},   Wexler almost periodicity of $X-$valued piecewise continuous function $\phi$ is equivalent 
to the Bohr almost periodicity of some associated $L_1([0,1],X)$- valued function 
(i.e. is equivalent to  the Stepanov almost periodicity of $\phi$).  

The following assertion shows how the almost periods for the triple consisting from a W-almost periodic function, 
an almost periodic sequence and a strongly almost periodic subset of $\R$ can be harmonized:    
\begin{lemma} \label{wep} Consider a strongly almost periodic set  $\{\tau_j\}$ and suppose that $\tau_{j+1} -\tau_j > 4\theta$ for some positive $\theta$. Let also the sequence 
$\{B_j\}$ of elements of some Banach space $Y$ be almost periodic and let  
$f: \mathbb{R} \to X$ be W-almost periodic  function.  Then 
for any $\epsilon > 0$  there exists $l >0$ such that  each subinterval of \ $\R$ of the length $l$ contains some integer $q$ and a real number $r$  such that, for all $k \in \mathbb{Z}$, 
$$|B_{k+q} - B_k| < \epsilon, \ |(\tau_{k+q}-\tau_{k}) - r| < \epsilon,\ \mbox{and} \quad |f(t+r) - f(t)| < \epsilon,\ t\in \R, \  |t - \tau_j| > \epsilon,\  j \in \mathbb{Z}.$$
\end{lemma}
\begin{proof} Consider $\phi(t)= \max\{0, 1-|t|/\theta\}$ and the following $Y$-valued `saw' function 
$$
F_1(t) = \sum_{j\in \Z} \phi(t-\tau_j)B_j, \ \tau_j = aj +c_j. 
$$
Note that for every fixed $t \in \mathbb{R}$, the sum above contains at most one non-zero term.
Take  any positive  $\epsilon$ smaller than  $\theta.$ Since the sequence $\{c_j,B_j\}$ of elements of $\R\times Y$ is almost periodic, there exists a relatively 
dense set $\mathcal T$ of integer $\epsilon-$almost periods for this sequence. If $p \in \mathcal T$, then  
\begin{eqnarray}
& & |F_1(t-ap)- F_1(t)| = \left| \sum_{j\in \Z} \phi(t-ap-\tau_j) B_j - \sum_{j\in \Z} \phi(t-\tau_{j+p})B_{j+p} \right| \nonumber\\
& & \le \sum_{j\in \Z} \phi(t-ap-\tau_j)|B_j-B_{j+p}|+  \sum_{j\in \Z} |\phi(t-ap-\tau_j)-\phi(t-\tau_{j+p})||B_{j+p}| \nonumber\\ 
& & < \epsilon(1+\sup_{j \in \Z}|B_j|/\theta). \label{f1}
\end{eqnarray}
Observe here that $|(t-ap-\tau_j)-(t-\tau_{j+p})| < \epsilon$ and therefore, for each fixed $t \in \mathbb{R},$  at most one term $|\phi(t-ap-\tau_j)-\phi(t-\tau_{j+p})|$ is different from zero. In addition, 
$$
|\phi(t-ap-\tau_j)-\phi(t-\tau_{j+p})| = \left|\int_{t-ap-\tau_j}^{t-\tau_{j+p}}\phi'(s)ds\right|< \frac\epsilon\theta. 
$$
By (\ref{f1}) the function $F_1$ is Bohr almost periodic.  Similarly, $F_2(t) =  \sum_{j\in \Z} \phi(t-\tau_j)$ is a scalar Bohr almost periodic function 
so that 
$F(t) = (F_1(t),F_2(t), f(t+\cdot))$ is $Y\times \R\times L_1([0,1],X)$-valued almost periodic function. 
Assume that $r$ is a $\delta-$ almost period of $F$.  Then we obtain
$$\int_t^{t+1}|f(s+r)-f(s)|ds< \delta, \quad t \in \R,$$
that implies, in view of \cite[Lemma 5]{Tr}, that, for some positive $\omega(\delta), \ \omega(0+)=0$, it holds
 $$|f(t+r) - f(t)| < \omega(\delta),\ t\in \R, \  |t - \tau_j| > \omega(\delta),\  j \in \mathbb{Z}.$$

Take some positive $\delta < \epsilon/(1+\sup_{j \in \Z}|B_{j}|+\theta)$ such that $\omega(\delta) < \epsilon$. Clearly,  $r$ is also $\delta-$ almost period of $F_2$ and therefore there exists 
integer $q$ such that $|(\tau_{k+q}-\tau_{k}) - r| < \delta\theta < \epsilon$ for all $k \in \Z$. In addition, $ |B_j- B_{j+q}|< \epsilon$ for all $j \in \Z$ since
 \begin{eqnarray*}
 & & \delta > |F_1(\tau_j)-F_1(\tau_j+r) |= |B_j- \phi(\tau_j+r-\tau_{j+q})B_{j+q}|  
 =|(B_j- B_{j+q}) \\
 & & + \frac{|\tau_j+r-\tau_{j+q}|}{\theta}B_{j+q}| \geq |B_j- B_{j+q}|- 
 \frac{|\tau_j+r-\tau_{j+q}|}{\theta}|B_{j+q}|\geq  |B_j- B_{j+q}|- \delta \sup_{j \in \Z}|B_{j}|.
  \end{eqnarray*}
Finally,   the conclusion of the lemma follows from the estimate $|r-qa| = |(\tau_{k+q}-\tau_{k})-qa - ((\tau_{k+q}-\tau_{k})-r) |< \epsilon + 2\sup|c_j|$ and  the relative density  on $\R$ of  the set of all $\delta$-almost periods $r$ of $F$.    \hfill $\square$
\end{proof}
\subsection{Semilinear impulsive systems in abstract spaces and the beating phenomenon.} \noindent 
Throughout the paper, given $\alpha \geq 0$ and $\rho >0$, we will assume the following hypotheses:  

\vspace{1mm}

\noindent ${\bf (H1)}$ $A:D(A)\subseteq X \to X$ is a sectorial operator and $\inf\,\{ Re \,\mu: \ \mu \in \sigma(A)\} \ge \delta > 0,$
where $\sigma(A)$ denotes the spectrum of $A.$
Consequently, the fractional powers of $A$ are well defined as well as the spaces  $X^\alpha = D(A^\alpha)$ endowed 
with the norms $|x|_\alpha = |A^\alpha x|.$   We set $U^\alpha_\rho = \{ x \in X^\alpha: \ | x|_\alpha \le \rho\}.$
\vspace{1mm}

\noindent  ${\bf (H2)}$ The function $A_1: \mathbb{R} \to L(X^\alpha,X)$ is Bohr almost periodic and Lipschitz continuous. 
\vspace{1mm}

\noindent  ${\bf (H3)}$ The functions $\tau_j: U^\alpha_\rho \to \mathbb{R}$ are such that $\tau_j(x)= aj+c_j(x)$, where $a>0$ and the sequence of continuous functions $\{c_j(x)\}$ is  almost periodic 
uniformly with respect to $x \in U^\alpha_\rho$.  Moreover,  there exists $\theta > 0$ such that $\inf_{U^\alpha_\rho} \tau_{j+1}(x) -
\sup_{U^\alpha_\rho} \tau_{j}(x) \ge \theta > 0,$ for all $j \in \mathbb{Z}.$ Thus,  for all $j \in \mathbb{Z}$, 
$$\sup_{U^\alpha_\rho}\tau_{j+1}(x) - \inf_{U^\alpha_\rho}\tau_j(x) \leq  \sup\{c_{j+3}(0)- c_{j}(0)\}+3a-2\theta:=\mathcal{Q}.$$ 

\noindent  ${\bf (H4)}$  The sequence $\{g_j(x)\}$ of continuous functions $g_j:U^\alpha_\rho \to X^\alpha$ is almost periodic uniformly with respect
to $x \in U^\alpha_\rho.$ In addition, $g_j(U^\alpha_\rho) \subset X^1$ for all $j \in \Z$. 
\vspace{1mm}

\noindent  ${\bf (H5)}$  The function $f: \ \mathbb{R} \times U^\alpha_\rho \to X$ is locally Lipschitzian   and is Bohr almost periodic in $t$
uniformly with respect to $x \in U^\alpha_\rho$. 
\vspace{1mm}

Under these hypotheses, $-A$ is the  infinitesimal generator of the analytic semigroup $S(t)=e^{-At}$ and 
 $e^{-At} A^\alpha x = A^\alpha e^{-At} x$ when  $x \in X^\alpha, t >0$.  The following inequalities are also valid  \cite{H}:
\begin{eqnarray*}
& & | A^\alpha e^{-A t} | \le C_\alpha t^{-\alpha} e^{-\delta t}, \ t > 0, \ \alpha > 0, \\[2mm]
& & | (e^{-A t} - I)x | \le \frac{1}{\alpha}C_{1-\alpha} t^\alpha |  A^\alpha x|, \quad t > 0,\ \alpha \in (0,1],\ x \in X^\alpha,\
\end{eqnarray*}
where $C_\alpha > 0$ is  bounded when $\alpha \to 0+.$ 

\begin{definition}
We say that  $u: [a, b] \to X^\alpha$ solves the initial value problem $u(a) = x_0 \in X^\alpha$
for the system (\ref{ban1}), (\ref{ban2}) on $[t_0, t_1]$ if $u(t)$ satisfies the initial condition and there 
exists a maximal  finite set of numbers $T_0=a< T_1<T_2< \dots <T_m =b$ such that 

(i) $u(t)$ is uniformly continuous on each interval $(T_j, T_{j+1})$ and $\tau_{m(j)}(u(T_j-))= T_j$ \ for $j= 1,\dots, m-1$ and appropriate $m(j)$
(i.e. the trajectory $(t,u(t))$ hits the union $\cup \Gamma_j$ of surfaces $ \Gamma_j: =\{(t,x): t= \tau_j(x), \ x \in U^\alpha_\rho\}$  {only} at the moments $T_j$).

(ii) $u: [a, b] \to X$ is continuously differentiable on each of the intervals $(T_0, T_1), \dots (T_{m-1}, T_m)$
and satisfies the equations (\ref{ban1}) and (\ref{ban2}) if $t \in (a,b), t \not= T_j$ and $t = T_j$,
respectively.

We assume that $u(t)$ is  left continuous so that  $u(T_j) = u(T_j - 0)$.
\end{definition}
Note  that each surface $\Gamma_j$ separates the cylinder $\R \times U^\alpha_\rho$ into two open parts, 
$$
W_j^+ =\{(t,x): t> \tau_j(x),  x \in U^\alpha_\rho\} , \quad W_j^- =\{(t,x): t< \tau_j(x),  x \in U^\alpha_\rho\}, 
$$
so that 
each trajectory $(t,u(t))$ either remains in one of these parts or intersects (hits) $\Gamma_j$ at least one time. We will assume 
that each trajectory intersects a surface $\Gamma_j$ at most one time. It implies that the beating phenomenon 
is excluded from our analysis. For impulsive systems in finite-dimensional spaces,  
there are several conditions allowing to
control the number of intersection of trajectories with $\Gamma_j$, see  \cite{LBS,SP,SPT}.  
Some of them can be extended for the abstract parabolic evolution systems: 
\begin{lemma} \label{L3v} In addition to the smoothness conditions imposed in  ${\bf (H2)}$ - ${\bf (H5)}$, 
suppose that  the function $F: \ \mathbb{R} \times X^1 \to X^{1-\alpha}, \ F(t,x)=-A_1(t)x+f(t,x)$, is well defined and 
locally Lipschitzian. Assume also  that, for some fixed $j$ and all $x \in U^\alpha_\rho\cap X^1$, it holds 
$$\theta_j(x):= \tau_j(x+g_j(x)) - \tau_j(x) \leq 0.$$ 
If there exists a differentiable extension $\tau_j: U^0_\varrho \to \R$, $\varrho = \sup_{U^\alpha_\rho}|x|$, of
the mapping $\tau_j: U^\alpha_\rho \to \mathbb{R}$  such that ${\partial\tau_j}/{\partial x}: U^\alpha_\rho\cap X^1\to X^*$ is 
continuous when $X^*$ is provided with weak-* topology and 
\begin{equation}\label{pd}
\frak{P}(x):=\frac{\partial\tau_j(x)}{\partial x}(-Ax+F(\tau_j(x), x))<1, \quad \mbox{for all} \ x \in U^\alpha_\rho\cap X^1, 
\end{equation}
then the graph of each solution $u(t)$  of system (\ref{ban1}), (\ref{ban2}) can intersect the surface $\Gamma_j$ at most once. 
\end{lemma}
\begin{proof} Below, we invoke  the following result from  \cite[Section 3.5, Exercise 1]{H}: suppose that $A$ is sectorial  and 
$F: \mathbb{R} \times X^1 \to X^{1-\alpha}$ is Lipschitzian in a neighborhood $(t_0, x_0)$ for some $\alpha \in [0,1)$. If $x_0 \in X^1,$
then there exists a unique solution $x(t)$ of the initial value problem  $x'(t) + Ax = F(t,x),$ $t > t_0, \ x(t_0) = x_0,$ 
on some interval $t_0 \le t \le t_1$ and $t \to x(t) \in X^1$ is continuous.

Now, suppose that the graph of some solution $u(t)$ to  (\ref{ban1}), (\ref{ban2}) intersects the surface $\Gamma_j$ at 
the point $(t_0, x_0)$ where  $t_0=\tau_j(x_0)$. Then $x_0, x_0+g_j(x_0)  \in U^\alpha_\rho\cap X^1$ and  inequality
$\tau_j(x_0+g_j(x_0)) \leq t_0$ implies that the point $(t_0, x_0+g_j(x_0))$ either belongs to $W_j^+$ or it lies on $\Gamma_j$.  
Since the map $F: \ \mathbb{R} \times X^1 \to X^{1-\alpha}$ is locally Lipschitzian, the  solution $u:[t_0,t_0+\nu] \to X^1$ of the initial 
value problem $u(t_0)= x_0+g_j(x_0)$ for  (\ref{ban1}) is continuous for some small positive $\nu$. 
Now,  if  $\tau_j(x_0+g_j(x_0)) = t_0$ then,  in virtue of our smoothness assumptions,  $\zeta(t):=t- \tau_j(u(t))$ is continuously differentiable 
on $(t_0,t_0+\nu]$ and $\zeta(t_0+) =0$. Moreover, using (\ref{pd}), we obtain that $\zeta'(t_0+) >0$. Hence,  $\zeta(t) >0$ and $(t,u(t)) \in W_j^+$ for small positive $t-t_0$. 
Obviously, the same happens when $\tau_j(x_0+g_j(x_0)) < t_0$.  Now, suppose that there exists  finite $t_1>t_0$ such that  $\zeta(t) >0$
for $t\in (t_0,t_1)$ and $\zeta(t_1)=0$.  Then $u(t_1)\in U^\alpha_\rho\cap X^1$ and therefore (\ref{pd}) yields  
$\zeta'(t_1) >0$, a contradiction. Thus $\zeta(t) >0$ for all admissible  $t>t_0$ that proves the lemma. 
\hfill $\square$
\end{proof} 
\begin{example}\label{EXA4} Consider the diffusive logistic equation with impulses at state dependent moments \begin{eqnarray}
& & u_t = u_{\xi\xi} + a(t)u(1 - b(t)u), \quad t >0,\ t \not= \tau_j(u(t,\xi)), \ \xi\in (0,l), \label{ex1} \\[2mm]
& &  u(t+0,\xi) =u(t,\xi) + \int_0^lK_j(\xi,\zeta)I_j(u(t,\zeta))d\zeta+ d_j(\xi),\quad  t = \tau_j(u(t,\xi)), \label{ex2} 
\end{eqnarray}
subjected to the  boundary conditions
$u(t,0) = u(t,l) = 0.$
Here $a,b:\R\to \R_+$ are bounded Lipschitz continuous functions and the sequence of surfaces $t=\tau_j(u)$ 
is defined by
$$
\tau_j(u) = t_j +  b_j \int_0^l u^2(\xi) d\xi, \quad b_j \leq 0, 
$$
where   real numbers $t_j$ satisfy  $t_{j+1} - t_j \ge \theta_* > 0$ for some $\theta_*$. 
We will assume that $d_j \in  H^2(0,l)\cap H^1_0(0,l),$ $d_j \geq 0,$ and that $K_j:[0,l]^2 \to \R_+$ are $C^2$-smooth functions such that $K_j(0,\zeta)=K_j(l,\zeta)=0$ for all $\zeta \in [0,l]$.  Next, we assume that $I_j:\R\to \R_+, \ I_j(0)=0,$ are globally Lipschitzian functions, with the Lipshchitz constants $\Lambda_j$. The above positivity  assumptions are related to the possible biological interpretation of  the quantity $u(t, x)$ as the 
number of individuals of a single species population per unit area at point $x$ and time $t$, cf. \cite{RogovY}.
Now, for $\rho>0$, consider $$X = L_2(0,l), \quad A = - \frac{\partial^2}{\partial \xi^2}, \quad A_1(t)x = -a(t)(1-\rho b(t))x, \quad X^1 = D(A) = H^2(0,l)\cap H^1_0(0,l).$$
It is well known (e.g. see \cite{H}) that 
the operator $A:D(A) \subset X \to X$ is sectorial and 
$X^{1/2} = D(A^{1/2}) = H^1_0(0,l).$ Hence,  it is convenient to take $\alpha =1/2$. Then $F(t,u):= a(t)u(1 - b(t)u)$ maps $\R\times X^1$ into $X^{1/2}$ and is locally Lipschitizian. We also have that 
${\partial\tau_j(u)}/{\partial u}=   2b_j u,$ so that 
${\partial\tau_j}/{\partial u}: X\to X^*=X$ is obviously continuous (even we consider the norm topology in $X^*$). 
For every  initial condition $u(t_0,\xi) = u_0(\xi)$, $u_0 \in H^1_0(0,l)$, the proposed  boundary value problem for equation  (\ref{ex1}) has a unique local classical solution $u(t,\xi)$. By the maximum principle, $u_0(\xi) \geq 0$ implies that $u(t,\xi)$ is a global solution and $u(t,\xi)\geq 0$ for all $t >t_0, \xi \in (0,l)$.  
Now, we claim that if  $b_j\leq 0$, then for each $\rho>0$ there exists $\beta_0= \beta_0(\rho,a(\cdot), b(\cdot))>0$ such that non-negative solutions $u(t,\xi)$ of (\ref{ex1}), (\ref{ex2}) with values in  $U^{1/2}_\rho$ can intersect each surface $\Gamma_j$ at most one time whenever $|b_j| \leq \beta_0, \ j \in \Z$ (clearly, $u(t,\xi)\geq 0$ if $u_0(\xi)\geq 0$). 
Indeed,  obviously $\theta_j(u) \leq 0$ for all $u\geq 0$ while 
$$
\frak{P}(u) = -2b_j \int_0^l u_\xi^2(\xi) d\xi+2b_ja(\tau_j(u))\int_0^lu^2(\xi)(1-b(\tau_j(u))u(\xi))d\xi \leq 
$$
$$
\leq 2\sup|b_j|(1+\sup [a(t)b(t)])(|u|^2_{1/2}+\sqrt{l}|u|^3_{1/2})\leq  2\sup|b_j|(1+\sup [a(t)b(t)])(\rho^2+\sqrt{l}\rho^3) <1
$$
whenever  $\sup_j |b_j| \leq \beta_0= 0.5((1+\sup [a(t)b(t)])(\rho^2+\sqrt{l}\rho^3))^{-1}$.   
\end{example}

\subsection{Linear impulsive systems in abstract spaces: exponential dichotomy and the Green function.}
\label{SGF}
\noindent 
Under the hypotheses assumed in the previous subsection and  for each  $0 \le \gamma < 1$,
the linear homogeneous equation
\begin{eqnarray} \label{lin1}
& & \frac{dx}{dt} + (A + A_1(t))x = 0
\end{eqnarray}
defines a strongly continuous family of evolution operators $U(t,s):X^\gamma \to X^\gamma, \   t \geq s,$  (hence, 
$U(\tau,\tau) = I, \ U(t,s)U(s,\tau) = U(t,\tau), \ t\ge s \ge \tau$).  See \cite[Theorem 7.1.3]{H} for more detail and the proof of the following inequalities
$$
| U(t,\tau)x|_\gamma \le C(t - \tau)^{(\nu - \gamma)_-}|x|_\nu, \quad 0\leq \nu  \leq 1,  \ x \in X^1, 
$$
\begin{eqnarray} \label{evop2}
| U(t,\tau)x - x|_\gamma \le C(t - \tau)^{\nu}|x|_{\gamma+\nu}, \quad \nu > 0,\   \gamma + \nu \le 1,\  x \in X^1, 
\end{eqnarray}
where  $t - \tau \le Q, \ C = C(Q,\nu,\gamma)$, and
$(\nu - \gamma)_- = \min(\nu - \gamma, 0).$
Suppose now that $u(t) \in X^\alpha, \ t \in \R,$ is a bounded solution of the equation (\ref{ban1}), then $u(t)$  solves   
linear inhomogeneous  equation obtained  from (\ref{lin1}) by replacing $0$ in the right-hand side with bounded perturbation 
$f(t)= f(t,u(t))$. It is well known that the exponential dichotomy is the right concept when analysing the relationship between $u(t)$ and $f(t)$: 

\begin{definition}\label{defn:dich} 
We say that  (\ref{lin1}) has an exponential dichotomy on $\mathbb{R}$ in the space $X^\alpha$
 with exponent $\beta > 0$ and bound $M \ge 1$ if there exist projections $P(t):X^\alpha\to X^\alpha$, $\frak{R}(t):= P(t)X^\alpha$, 
such that
\vspace{2mm}

(i) $U(t,s)P(s) = P(t)U(t,s), \ t \ge s$;
\vspace{2mm}

(ii) $U(t,s): \frak{R}(s) \to \frak{R}(t)$ is invertible for  $t \ge s$, with the inverse denoted as   
$U(s,t)$;
\vspace{2mm}

(iii) $|U(t,s)(I - P(s))x|_\alpha \le M e^{-\beta(t-s)}|x|_\alpha, \ t \ge s, \ x \in X^\alpha;$
\vspace{2mm}

(iv) $|U(t,s)P(s)x|_\alpha \le M e^{\beta(t-s)}|x|_\alpha, \ t \le s, \ x \in X^\alpha.$
\end{definition}
Due to \cite[Section 7.6]{H}, the exact value of $\alpha \in [0,1)$ is not relevant in the 
above definition:  equation (\ref{lin1}) has an exponential dichotomy in the space $X^\alpha$ if, and only if, it
has exponential dichotomy in the space $X^\gamma$ for any $0 \le \gamma < 1$.

To each exponentially dichotomic system, we associate the Green function
\begin{eqnarray*}
G(t,s) =
 \left\{\begin{array}{l}
            U(t,s)(I - P(s)), \ t > s,\\
            -U(t,s)P(s), \ t \leq s. 
          \end{array} \right.
\end{eqnarray*}
Clearly, $
|G(t,s)x|_\alpha \le M e^{-\beta|t-s|}|x|_\alpha, \ t,s \in \mathbb{R}. 
$
In fact, this estimate admits the following extension (see  \cite[Lemma  7.6.2]{H}): 
there is a real number $M_1\geq M$ (depending on $\alpha, \delta, \gamma$) such that
\begin{eqnarray} \label{inlin6}
|U(t,s)P(s)x|_\gamma \le M_1 e^{-\beta(s-t)}|x|_\delta, \ s \ge t,
\end{eqnarray}
when $0 \le \gamma < 1, \ \delta \ge 0,$ and, for $0 \le \delta \le \gamma < 1,\ t >s,$ $\psi_{\gamma-\delta}(t-s):= 1+(t-s)^{\delta-\gamma}$, 
\begin{eqnarray} \label{inlin7}
|U(t,s)(I - P(s))x|_\gamma \le M_1 e^{-\beta(t-s)} \psi_{\gamma-\delta}(t-s)|x|_\delta, \quad t > s. 
\end{eqnarray}
It is also convenient to set $\psi_{\gamma}(s) =1$ for $s \leq 0$. 
%
 
Consequently, if $f:\R\to X$ is a locally H\"older continuous bounded function  and sequence $\{g_j\}, \ g_j \in X,$ is bounded, 
then  the following  left-continuous function is well defined:
\begin{equation} \label{inlin3}
u_0(t) = \int_{-\infty}^{\infty}G(t,v)f(v)dv + \sum_{j \in \mathbb{Z}}G(t, \bar\tau_j)g_j,\quad \mbox{where} \ \bar\tau_j = \tau_j(0).
\end{equation}
A straightforward verification shows  that $u(\bar\tau_j + 0) - u(\bar\tau_j) =  g_j$ and that, for each  pair of points $s<t$ taken in  
$(\bar \tau_j, \bar \tau_{j+1})$,  it holds 
$$ 
u_0(t) = U(t,s)u_0(s) + \int_s^tU(t,v)f(v)dv
$$
and therefore $u_0(t)$ satisfies  the both equations of the impulsive system
\begin{eqnarray} \nonumber
& & \frac{dx}{dt} + (A + A_1(t)) x = f(t), \quad t \not= \bar\tau_j, \\[2mm]
& & x(\bar\tau_j + 0) - x(\bar\tau_j) =  g_j, \quad j \in \mathbb Z. \nonumber
\end{eqnarray}
Actually $u_0:\R\to X$ is the unique bounded solution of  the above system: if $u_1(t)$ were another bounded 
solution of this impulsive system, then continuous function $u_*(t)=u_1(t)-u_0(t)$
(observe that $u_*(\bar\tau_j + 0)=u_*(\bar\tau_j)$ since $u_1(\bar\tau_j + 0) - u_1(\bar\tau_j) = u_0(\bar\tau_j + 0) - u_0(\bar\tau_j)=  g_j$)
would be a non-zero bounded solution of the  
exponentially dichotomic homogeneous equation (\ref{lin1}) that clearly is not possible.   
 
In the next section, we will use the following property of the Green function: 
\begin{lemma}  \label{in2} Suppose that  $0 \le \delta \le \alpha < 1$, $h \in \R$. 
Then 
$$
|(G(t+h,\tau+h) - G(t,\tau))x|_\alpha \le M_2 e^{-\beta_1|t-\tau|}\psi_{\alpha-\delta}(t-\tau) a_*(h)|x|_\delta, \ x \in X^\delta,\ t\not=\tau,
$$
with some positive constants $\beta_1 \le \beta, M_2 \ge M$ and $a_*(h) = \sup_{s\in\R}|A_1(s) - A_1(s+h)|_{L(X^\alpha,X)}.$
\end{lemma}
\begin{proof}
Clearly, for each fixed $h \in \mathbb{R}$, $G(t+h,s+h)$ is the Green function for  the shifted equation
$u'(t) + (A + A_1(t+h))u(t) = 0.$ Consequently,  for each compactly supported locally H\"older continuous function $\xi:\R\to  X$,  the linear inhomogeneous equation
$u'(t) + (A + A_1(t+h))u(t) = \xi(t)$ has a unique bounded solution $u(t,\xi) = \int_{-\infty}^\infty G(t+h,s+h)\xi(s) ds$.  Since $u(t,\xi)$ is also a unique bounded solution of the  equation $u'(t) + (A + A_1(t))u(t) = r(t,\xi)$ with locally H\"older continuous bounded function  $r(t,\xi)= \xi(t)+(A_1(t)-A_1(t+h))u(t,\xi)$, we find that 
$$
u(t,\xi) = \int_{-\infty}^\infty G(t+h,s+h)\xi(s) ds = \int_{-\infty}^\infty G(t,s)r(s,\xi)ds. 
$$
Thus, invoking the Fubini theorem for  abstract integrals \cite{HPh}, we obtain that 
$$
0= \int_{-\infty}^\infty \left[(G(t+h,s+h)-G(t,s)) \xi(s)-G(t,s)(A_1(s)-A_1(s+h)) \int_{-\infty}^\infty G(s+h,u+h)\xi(u) du \right]ds
$$
$$
= \int_{-\infty}^\infty \left[G(t+h,s+h)-G(t,s)-\int_{-\infty}^\infty G(t,u)(A_1(u)-A_1(u+h)) G(u+h,s+h)du\right]\xi(s)ds.
$$
Since compactly supported locally H\"older continuous  function $\xi(s)$ is arbitrary,  this implies that   
\begin{eqnarray} \label{Ge}
G(t+h,\tau+h) - G(t,\tau) = \int_{-\infty}^{\infty}G(t,s)(A_1(s) - A_1(s+h))G(s+h,\tau+h)ds.
\end{eqnarray} 
Note that relation (\ref{Ge}) was proposed as an exercise  in \cite[Section 7.6]{H}. 
Hence, in view of  (\ref{inlin6}), (\ref{inlin7}), we obtain, for  $t < \tau$,  the following estimates for the right-hand side of (\ref{Ge}): 
\begin{eqnarray*} 
& &  \hspace{-7mm} |(G(t+h,\tau+h) - G(t,\tau))x|_\alpha \le
\Biggl(\int_{\tau}^{\infty} e^{-\beta(s-t)} e^{-\beta(s-\tau)}\psi_{\alpha-\delta}(s-\tau)ds  \\
& & \hspace{-7mm} + \int^{\tau}_{t}e^{-\beta(s-t)} e^{-\beta(\tau-s)}ds +
\int_{-\infty}^t e^{-\beta(t-s)} \psi_{\alpha}(t-s)e^{-\beta(\tau-s)}ds\Biggl)M_1^2 a_*(h) |x|_\delta \\
& & \hspace{-10mm}=\Biggl(e^{\beta(t+\tau)}\int_{\tau}^{\infty} e^{-2\beta s}\psi_{\alpha-\delta}(s-\tau)ds +
e^{-\beta(\tau - t)}(\tau - t) + e^{-\beta(t+\tau)}\int_{-\infty}^t e^{2\beta s} \psi_{\alpha}(t-s)ds\Biggl)M_1^2 a_*(h) |x|_\delta  \\
& & \le
M_2 e^{-\beta_1(\tau-t)}a_*(h) |x|_\delta.
\end{eqnarray*}
Similarly, if $t > \tau$ then
\begin{eqnarray*}
& & |(G(t+h,\tau+h) - G(t,\tau))x|_\alpha \le
\Biggl(\int_{t}^{\infty} e^{-\beta(s-t)} e^{-\beta(s-\tau)}\psi_{\alpha-\delta}(s-\tau)ds  \\
& & + \int_{\tau}^{t}e^{-\beta(t-s)}\psi_{\alpha}(t-s) e^{-\beta(s-\tau)}\psi_{\alpha-\delta}(s-\tau)ds  + \int_{-\infty}^{\tau} e^{-\beta(t-s)} \psi_{\alpha}(t-s)e^{-\beta(\tau-s)}ds\Biggl)M_1^2 a_*(h) |x|_\delta  \\
& &  =\Biggl(e^{\beta(t+\tau)}\int_{t}^{\infty} e^{-2\beta s}\psi_{\alpha-\delta}(s-\tau)ds +
e^{-\beta(t-\tau)}\int_{\tau}^{t}\psi_{\alpha}(t-s) \psi_{\alpha-\delta}(s-\tau)ds  \\
& & + e^{-\beta(t+\tau)}\int_{-\infty}^{\tau} e^{2\beta s} \psi_{\alpha}(t-s)ds\Biggl)M_1^2 a_*(h) |x|_\delta \le
M_2 e^{-\beta_1(t-\tau)}\psi_{\alpha-\delta}(t-\tau)a_*(h) |x|_\delta.
 \end{eqnarray*}
 Observe that, in order to estimate the integral $\int_{\tau}^{t}\psi_{a}(t-s) \psi_{b}(s-\tau)ds$, we are using the relation 
 $\int_{\tau}^{t} (t-s)^{-a} (s-\tau)^{-b}ds = (t-\tau)^{1-a-b}B(1-b,1-a)$ where $B$ is the beta function.  The above estimations also justify the use of the Fubini theorem. 
 \hfill $\Box$
\end{proof}


\section{Almost periodic solutions of abstract parabolic differential equation with impulses}
\noindent At first, we study linear almost periodic impulsive systems:
\begin{thm}\label{T7}
Assume  the hypotheses ${\bf (H1)}$ -- ${\bf (H4)}$ (where $g_j$ are constants)  and  also suppose that 
 function $f: \mathbb{R} \to X$ is W-almost periodic and locally H\"{o}lder continuous with points of discontinuity
at $t = \bar\tau_j, j \in \mathbb{Z}$. 
If equation (\ref{lin1}) possesses an exponential dichotomy on $\R$, then the unique bounded solution $u_0:\R \to X$ is  W-almost periodic as map $u_0:\R \to X^\alpha.$
\end{thm}
\begin{proof}  It suffices to show that the bounded  solution $u_0(t)$ given  in (\ref{inlin3})
is W-almost periodic as $X^\alpha$-valued function. Given positive $\varepsilon$, take $\varepsilon$-periods $r, q$ for $A_1: \R \to L(X^\alpha,X), \ f:\R\to X$ and $\{g_j\}$ as in  Lemma \ref{wep} and observe that  
$$
|u_0(t + r) - u_0(t)|_\alpha =  \left|\int_{-\infty}^\infty (G(t + r,s)- G(t,s))f(s)ds + \sum_{j \in \mathbb{Z}} (G(t + r, \bar\tau_{j})- G(t, \bar\tau_{j}))g_{j} \right|_\alpha. 
$$
Since $|f(t+r) - f(t)| < \varepsilon$ for $|t - \bar\tau_j| \ge \varepsilon$, in virtue of Lemma \ref{in2}, we find that, for $t\in \R$, 
\begin{eqnarray}
& & \left|\int_{-\infty}^\infty  (G(t + r,s)- G(t,s))f(s)ds\right|_\alpha \le  \int_{-\infty}^\infty |(G(t + r,s+r) - G(t,s))f(s+r)|_\alpha ds  \nonumber\\
& &  + \sum_{k \in \mathbb{Z}} \int_{\bar\tau_k + \varepsilon}^{\bar\tau_{k+1} - \varepsilon} |G(t,s)(f(s+r) - f(s))|_\alpha ds +
\sum_{k \in \mathbb{Z}} \int_{\bar\tau_k - \varepsilon}^{\bar\tau_{k} + \varepsilon} |G(t,s)(f(s+r) - f(s))|_\alpha  ds \nonumber\\
& &  \le \varepsilon M_2 \sup_{s \in \R}|f(s)| \int_{-\infty}^\infty e^{-\beta_1|t-s|}\psi_\alpha(t-s)ds +
\varepsilon M_1 \int_{-\infty}^\infty e^{-\beta|t-s|}\psi_\alpha(t-s)ds  \nonumber\\
& &  + 2 M_1 \sup_{s \in \R}|f(s)|\sum_{k \in \mathbb{Z}} \int_{\bar\tau_k - \varepsilon}^{\bar\tau_{k} + \varepsilon}
 e^{-\beta|t-s|}\psi_\alpha(t-s)ds:= \Gamma_1(\varepsilon), \quad \lim_{\varepsilon \to 0+}\Gamma_1(\varepsilon) = 0.  \label{in5}
\end{eqnarray}
Let $\bar\tau_{m-1} + \varepsilon \le t \le \bar\tau_{m} - \varepsilon.$ Then $\bar\tau_{m +q-1}  < t + r < \bar\tau_{m + q}$ and, using (\ref{evop2}), (\ref{inlin6}) and (\ref{inlin7}),  we obtain
\begin{eqnarray*}
& & \frak{S}:= \sum_{j \in \mathbb{Z}} | (G(t + r,\bar\tau_j + r) - G(t + r,\bar\tau_{j + q}))g_{j+q}|_\alpha  \nonumber\\ 
& &
\le \sum_{ j \ge m, \bar\tau_{j+q} \geq \bar\tau_j +r} |(U(t + r, \bar\tau_{j+q}) P(\bar\tau_{j+q})(I - U(\bar\tau_{j+q}, \bar\tau_j +r))g_{j+q}|_\alpha  \nonumber\\
& &
+ \sum_{ j \ge m, \bar\tau_{j+q} < \bar\tau_j +r} |(U(t + r, \bar\tau_{j}+r) P(\bar\tau_{j}+r)(I - U(\bar\tau_j +r, \bar\tau_{j+q}))g_{j+q}|_\alpha  \nonumber\\
& &   +\sum_{ j < m,  \bar\tau_{j+q} \geq \bar\tau_j +r} |(U(t + r, \bar\tau_{j+q})(I - P(\bar\tau_{j+q}))(I - U(\bar\tau_{j+q}, \bar\tau_j +r))g_{j+q}|_\alpha  \nonumber\\
& &  + \sum_{ j < m,  \bar\tau_{j+q} < \bar\tau_j +r} |(U(t + r, \bar\tau_{j}+r)(I - P(\bar\tau_{j}+r))(I - U(\bar\tau_j +r, \bar\tau_{j+q}))g_{j+q}|_\alpha   \nonumber\\
\end{eqnarray*}
\begin{eqnarray*}
& & \le \sum_{\bar\tau_{j+q} \geq \bar\tau_j +r} M e^{-\beta|\bar\tau_{j+q} - t - r|}|(I - U(\bar\tau_{j+q}, \bar\tau_j +r))g_{j+q}|_\alpha  +  \sum_{\bar\tau_{j+q} < \bar\tau_j +r} M e^{-\beta |t - \bar\tau_{j}|}|(I - U(\bar\tau_j +r, \bar\tau_{j+q}))g_{j+q}|_\alpha  \nonumber\\
& & \le \frac{2M C}{1 - e^{-\beta\theta}}\varepsilon^{1-\alpha}\sup_j|g_{j}|_1 +
\frac{2M C}{1 - e^{-\beta\theta}}\varepsilon^{1 - \alpha}\sup_j|g_{j}|_{1} =:
\Gamma_2(\varepsilon)\sup_j|g_{j}|_{1}, \nonumber
\end{eqnarray*}
where $\Gamma_2(\varepsilon) \to 0$ as $\varepsilon \to 0+.$ Consequently, 
 \begin{eqnarray}
& & \left|\sum_{j \in \mathbb{Z}} (G(t+r, \bar\tau_{j})- G(t, \bar\tau_{j}))g_{j} \right|_\alpha =  \left|\sum_{j \in \mathbb{Z}} G(t, \bar\tau_{j})g_{j} -
\sum_{j \in \mathbb{Z}} G(t + r, \bar\tau_{j+q})g_{j+q}\right|_\alpha  \nonumber\\
& & \le \sum_{j \in \mathbb{Z}} |G(t,\bar\tau_j)(g_{j+q} - g_j)|_\alpha +
\sum_{j \in \mathbb{Z}} | (G(t, \bar\tau_{j}) - G(t + r,\bar\tau_j + r))g_{j+q}|_\alpha + \frak{S}  \nonumber\\
& & \le \left( \frac{2M}{1 - e^{-\theta\beta}} + \frac{2M_2}{1 - e^{-\theta\beta_1}}\sup_j|g_j|_\alpha\right)\varepsilon +
\Gamma_2(\varepsilon)\sup_j|g_{j}|_{1}. \label{in7}
\end{eqnarray}
Finally, due to (\ref{in5}) and (\ref{in7}), $
|u_0(t + r) - u_0(t)|_\alpha \le\Gamma_3(\varepsilon),
$
for $|t - \bar\tau_j| \ge \varepsilon$, 
where $\Gamma_3(\varepsilon) \to 0$ as $\varepsilon \to 0+.$ This proves Wexler's almost periodicity of the function $u_0: \R \to X^\alpha.$ \hfill $\Box$
\end{proof}
Set 
$$
K=K(\alpha):= K_1(\alpha)+K_2(\alpha), \quad \mbox{where} \ K_1(\alpha): = \int_{-\infty}^\infty M_1 \psi_\alpha(s)e^{-\beta|s|} ds, \  \ K_2(\alpha): = \frac{2M_1}{1 - e^{-\beta\theta}}.
$$
Now we are in the position to prove the main result of this work: 
\begin{thm} \label{Thm7}For fixed $\alpha \geq 0,\  \rho >0$, assume  the hypotheses ${\bf (H1)}$ -- ${\bf (H5)}$. Suppose also that  

1) in the set $\R\times U_\rho^\alpha$, each trajectory of (\ref{ban1}), (\ref{ban2})
meets every surface $\Gamma_j$ no more than once;

2) $| f(t_1, u) - f(t_2, u)| \le H_1 |t_1 - t_2|,  \  \ H_1 > 0, u \in U_\rho^\alpha;$

3) $| f(t, u_1) - f(t, u_2)| + |g_j(u_1) - g_j(u_2)|_{\alpha}+
 |\tau_j(u_1) - \tau_j(u_2)| \le N_1|u_1 - u_2|_\alpha, $ uniformly in $t \in \mathbb{R}, j \in \mathbb{Z}, u \in U_\rho^\alpha$;

3) $|f(t,0)| \le M_0,\ t \in \R,  \ |g_j(0)|_{1} \le M_0, \ j \in \mathbb{Z}$. Moreover,  $|g_j(u)|_{1} \le g_*$ for all $|u|_\alpha \leq \rho, \ j \in \mathbb{Z}$;

4) equation (\ref{lin1}) is exponentially dichotomous  and $K(\alpha)M_0 < \rho$. 

If, in addition, $N_1$ is sufficiently small,  then system  (\ref{ban1}), (\ref{ban2})
has a unique W-almost periodic solution with values in  the domain  $U_\rho^\alpha$.
\end{thm}
\begin{proof} Let $\mathfrak{N}$ denote the complete metric space of all almost periodic
sequences $y = \{y_j\},$ $y_j \in X^\alpha,$ $|y_j|_\alpha \leq \rho,$ provided with the distance
$\| y-z\| = \sup_{j}|y_j-z_j|_\alpha.$ For 
$y = \{y_j\} \in \mathfrak{N}$, 
consider the following equation with the fixed moments $\tilde\tau_i = \tau_i(y_i)$ of  impulse action:
\begin{eqnarray} \label{ban1y}
& & \frac{dx}{dt} + (A + A_1(t)) x = f(t, x), \quad t \not = \tilde\tau_j, \\[2mm]
 & & x(\tilde\tau_j + 0) - x(\tilde\tau_j) = g_j(y_j), \quad j \in \mathbb Z. \label{ban2y}
\end{eqnarray}
Take  $u_0(t,y) \equiv 0$ and consider the sequence 
\begin{equation}\label{ipr} 
u_{n+1}(t,y) = \int_{-\infty}^\infty G(t,s) f(s, u_n(s,y))ds + \sum_{j \in \mathbb{Z}}
G(t,\tilde \tau_j)g_j(y_j), \ n = 0,1,\dots
\end{equation}
This sequence is well defined since, for sufficiently small $N_1$ and all $n\geq 0$,  $$|u_{n+1}(t,y)|_\alpha \le
 \int_{-\infty}^\infty|G(t,s)(f(s, 0) + f(s, u_n(s,y)) - f(s, 0))|_\alpha ds $$
 $$+\sum_{j \in \mathbb{Z}} | G(t, \tilde\tau_j)(g_j(0) + g_j(y_j) - g_j(0))|_\alpha \le  
K_1 M_0 + K_1 N_1\sup_t|u_n(t,y)|_\alpha + K_2 M_0 + K_2 N_1\rho <\rho.$$ 
Now, Theorem \ref{T7} assures that $u_1:\R\to X^\alpha$ is an W-almost periodic function. Thus 
$f(t,u_1(t))$ is also $X$-valued W-almost periodic function. In consequence, Theorem \ref{T7}  yields 
that function $u_2:\R\to X^\alpha$ is  W-almost periodic. Repeating this procedure, we conclude that 
each function $u_{n+1}(\cdot,y): \R \to X^\alpha$ is W-almost periodic. In addition, the sequence $\{u_{n}(t,y)\}$ is 
converging, uniformly on $\R$,  to the limit function $u^*(t,y)$ since 
$$\sup_{t\in \R}|u_{n+1}(t,y)- u_{n}(t,y)|_\alpha\leq KN_1\sup_{t\in \R}|u_{n}(t,y)-u_{n-1}(t,y)|_\alpha.$$
Clearly, $u^*:\R \to X^\alpha$ is 
W-almost periodic function and it satisfies the  integral equation
\begin{equation}\label{nne}
u(t,y) = \int_{-\infty}^\infty G(t,s) f(s, u(s,y))ds + \sum_{j \in \mathbb{Z}}
G(t,\tau_j(y_j))g_j(y_j).
\end{equation}
As we have already shown, this means that $u^*(t,y)$ is the unique bounded solution of the system 
(\ref{ban1y}), (\ref{ban2y}) taking values in the domain  $U_\rho^\alpha\cap X^1$.  Note that 
$$
\sup_{y  \in \mathfrak{N}, s \in \R} |f(s, u^*(s,y))| \leq M_0+N_1\rho:=M_*, \quad \sup_{\{y_j\}  \in \mathfrak{N}} \|g_j(y_j)\| \leq M_0+N_1\rho. 
$$
Therefore, for a fixed $\gamma \in [\alpha,1)$, uniformly on $y  \in \mathfrak{N}$ and  $t \in \R\setminus\{\tilde\tau_j\}$, 
we have that 
\begin{eqnarray*}
 & & |u^*(t,y)|_\gamma \le \int_{-\infty}^\infty|G(t,s)f(s, u^*(s,y))|_\gamma ds \\
 & & + \sum_{j \in \mathbb{Z}} | G(t, \tilde\tau_j)g_j(y_j)|_\gamma \leq M_*(K_1(\gamma) 
+\sum_{j \in \mathbb{Z}} M_1e^{-\beta|t- \tilde\tau_j|}\psi_{\gamma-\alpha}(t- \tilde\tau_j)).
\end{eqnarray*}
In particular,  
\begin{equation}\label{vot}
|u^*(t,y)|_\gamma \le 
M_*(K_1(\gamma)   +K_2(1+(\theta/4)^{\alpha-\gamma})), \ t \in \R\setminus\ \cup_j[\tilde\tau_j, \tilde\tau_j+\theta/4), \ y  \in \mathfrak{N}.
\end{equation}
Since $u^*(t,y)$ is W-almost periodic and 
$ \tilde\tau_{j+1} -  \tilde\tau_j \geq \theta, \ j \in \Z$, it is easy to find (e.g., see \cite[p. 214]{SP}) that 
$X^\alpha$-valued sequence $S(y) := \{u^*(\tau_j( y_j), y), \ j \in \Z\}$ is also almost periodic, 
$S(y) \in \frak{N}$. 
Evidently,   $u^*(\cdot, y^*):\R \to U_\rho^\alpha$ will be required 
W-almost period solution of system (\ref{ban1}), (\ref{ban2}) 
if, and only if, 
$y^*\in \frak{N}$ is such that $S(y^*)= y^*$. 

We will find sufficient conditions for the  map $S: \frak{N} \to \frak{N}$ to be a contraction.  Take 
$z,y \in \frak{N}$ and set  $\tilde\tau_j^1 = \tau_j(y_j), \ \tilde\tau_j^2 = \tau_j(z_j),$ 
$\tau''_{j}= \max\{\tilde\tau_{j}^1,\tilde\tau_{j}^2\},\ \tau'_j= \min\{\tilde\tau_j^1,\tilde\tau_j^2\}. $
We have 
$\|S(y) - S(z)\| =\sup_j
| u^*(\tilde\tau_j^1, y) -
u^*(\tilde\tau_j^2, z)|_\alpha$ and we will estimate the difference $| u^*(\tilde\tau_j^1, y) -
u^*(\tilde\tau_j^2, z)|_\alpha$ for an arbitrary fixed $j$.  Without loss of generality, we can 
assume  that  $\tilde\tau_j^1 \leq  \tilde\tau_j^2$, for this particular $j$.  Then it is convenient to use the 
triangle inequality by adding and subtracting the same term $u^*(\tilde\tau_j^1, z)= u^*( \tau_j(y_j), z)$:
\begin{equation}\label{pe}
| u^*(\tilde\tau_j^1, y) -
u^*(\tilde\tau_j^2, z)|_\alpha \le | u^*(\tilde\tau_j^1, y) - u^*(\tilde\tau_j^1, z)|_\alpha +
|u^*(\tilde\tau_j^1, z) - u^*(\tilde\tau_j^2, z)|_\alpha. 
\end{equation}
At first, we will estimate the difference $|u^*(\tilde\tau_j^1, y) - u^*(\tilde\tau_j^1, z)|_\alpha$. At this step, the iterative 
process (\ref{ipr}) will be again invoked. In particular, for $n=0$ and $t \in (\tau''_m, \tau'_{m+1}]$, we get
\begin{eqnarray}
& & \hspace{-7mm} i_1:=|u_{1}(t,y) - u_{1}(t,z)|_\alpha \leq  \sum_k |G(t,\tilde\tau_k^1)\left(g_k(y_k) - g_k(z_k)\right)|_\alpha + \sum_k | \left(  G(t,\tilde\tau_k^1) -  G(t,\tilde\tau_k^2)\right) g_k(z_k)|_\alpha  \nonumber\\
& & \le \sum_{k} M e^{-\beta|t - \tilde\tau_k^1|}N_1|y_k-z_k|_\alpha +
\sum_{k} |G(t,\tau''_k)(U(\tau''_k, \tau'_k) - I)g_k(z_k)|_\alpha  \nonumber\\
& & 
\leq  \frac{2M N_1}{1 - e^{-\beta \theta}}\|y-z\| + \sum_{k} M_1 C e^{-\beta|t - \tau''_k|}\psi_\alpha(t-\tau''_k)|\tau''_k - \tau'_k|\,|g_k(z_k)|_1  \nonumber\\
 & & \le N_1 \|y-z\|\left(\frac{2}{1 - e^{-\beta \theta}} \left(M + M_1 C(1+\theta^{-\alpha})g_*\right) + (t-\tau''_m)^{-\alpha}g_*M_1 C\right).
 \nonumber
\end{eqnarray}
Similarly, 
\begin{eqnarray}
& & |u_{n+1}(t,y) - u_{n+1}(t,z)|_\alpha \nonumber\\
& &   =\left| \int_{\R} G(t,s) (f(s, u_n(s,y))-f(s, u_n(s,z)))ds +\sum_{k \in \mathbb{Z}}  (G(t,\tilde\tau_k^1)g_k(y_k) -
G(t,\tilde\tau_k^2)g_k(z_k))\right|_\alpha  \nonumber\\
& & \le \sum_k \int_{\tau''_k}^{\tau'_{k+1}} |G(t,s)\left( f(s, u_n(s,y)) -
f(s, u_n(s,z))\right)|_\alpha ds \nonumber\\
& & +\sum_k \int_{\tau'_k}^{\tau''_k} |G(t,s)\left( f(s, u_n(s,y)) -
f(s, u_n(s,z))\right)|_\alpha ds +  i_1 =: i_3+i_2+i_1.
 \nonumber
\end{eqnarray}
For $t \in (\tau''_m, \tau'_{m+1}]$, we obtain
\begin{eqnarray}
& & i_2 \leq  \sum_k \int_{\tau'_k}^{\tau''_{k}} |G(t,s)f(s, u_n(s,y))|_\alpha ds +
\sum_k \int_{\tau'_k}^{\tau''_{k}} |G(t,s)f(s, u_n(s,z))|_\alpha ds  \nonumber\\
& &  \le 2\sum_{k \not= m} M_1 e^{-\beta\min\{|t - \tau''_{k}|, |t -  \tau'_{k}|\}}(1+\theta^{-\alpha})M_* N_1 \|y-z\| +
\int_{\tau'_m}^{\tau''_{m}} |G(t,s)f(s, u_n(s,y))|_\alpha d s   \nonumber\\
& & +\int_{\tau'_m}^{\tau''_{m}} |G(t,s)f(s, u_n(s,z))|_\alpha ds \le \left(\frac{2 M_1(1+\theta^{-\alpha})}{1 - e^{-\beta \theta}} + 2 M_1(t - \tau''_m)^{-\alpha}\right)M_*N_1 \|y-z\|.
\nonumber 
\end{eqnarray}
Therefore, if $t \in (\tau''_m, \tau'_{m+1}]$, then 
$$
|u_{n+1}(t,y) - u_{n+1}(t,z)|_\alpha \le i_3+ \left(K_3 + \frac{K_4}{(t -  \tau''_m)^\alpha}\right)N_1 \|y-z\|,
$$
where
\begin{eqnarray*}
K_3 = \frac{M_1(1+\theta^{-\alpha})(1 +Cg_* + 2M_*)}{1 - e^{-\beta \theta}}, \quad K_4 = M_1\left(g_*C +
2M_*\right).
\end{eqnarray*}
Fix some $n\geq 1$ now and suppose that there are positive constants $L'_n$ and $L''_n$ such that 
$$|u_{n}(t,y) - u_{n}(t,z)|_\alpha \le \left(L'_n + \frac{L''_n}{(t -  \tau''_k)^\alpha}\right)N_1 \|y-z\|, 
\quad t \in (\tau''_k, \tau'_{k+1}], \ k \in \Z.$$
Then, for $t \in (\tau''_m, \tau'_{m+1}]$, we have that $ |u_{n+1}(t,y) - u_{n+1}(t,z)|_\alpha \le $
\begin{eqnarray}
& & \hspace{-3mm}N_1 \|y-z\| \left(K_3 + \frac{K_4}{(t -  \tau''_m)^\alpha}\right) + N_1^2\|y-z\| \sum_{k } \int_{\tau''_k}^{\tau'_{k+1}}M_1 e^{-\beta|t-s|} \psi_{\alpha}(t-s)\left( L'_n + \frac{L''_n}{(s - \tau''_k)^\alpha}\right)ds  \nonumber\\[2mm]
& & \leq  N_1 \|y-z\|\left(K_3 + \frac{K_4}{(t -  \tau''_m)^\alpha}\right)+N_1^2\|y-z\|\, \Biggl[\sum_{k} \int_{\tau''_k}^{\tau'_{k+1}}M_1 e^{-\beta|t-s|} (1+\theta^{-\alpha}) \left( L'_n + \frac{L''_n}{(s - \tau''_k)^\alpha}\right)ds \nonumber\\[2mm]
& &
+\int_{\tau''_{m}}^t M_1(t-s)^{-\alpha} \left( L'_n + \frac{L''_n}{(s - \tau''_m)^\alpha}\right)ds \Biggl] \  \nonumber\\[2mm]
& & \hspace{-5mm}\le \Biggl( 2\frac{(1+\theta^{-\alpha})}{1 - e^{-\beta\theta}}\left( L'_n\mathcal{Q} + \frac{L''_n \mathcal{Q}^{1-\alpha}}{1-\alpha}\right) +
{L''_n B(1-\alpha,1-\alpha)}(t -  \tau''_m)^{1-2\alpha} +\frac{L'_n}{1-\alpha}(t -  \tau''_m)^{1-\alpha}\Biggl)M_1N_1^2\|y-z\|   \nonumber\\[2mm]
& &   \hspace{-5mm} +
 \left(K_3 + \frac{K_4}{(t -  \tau''_m)^\alpha}\right)N_1 \|y-z\| \le  \left( L'_n N_1 \Psi_1 + L''_n N_1 \Psi_2 + L''_n \frac{N_1 \Psi_3}{(t-\tau''_m)^\alpha} +
 K_3 + \frac{K_4}{(t-\tau''_m)^\alpha}\right)N_1\|y-z\|  \nonumber\\[2mm]
 & & = \left(L'_{n+1} + \frac{L''_{n+1}}{(t - \tau''_m)^\alpha}\right)N_1 \|y-z\|, \nonumber
\end{eqnarray}
where $L'_{n+1}:= N_1\Psi_1L_n' + N_1\Psi_2L_n''+K_3,\ L''_{n+1}:= N_1\Psi_3L_n''+K_4, \ L'_0=L_0''=0,$
$$\Psi_1 = \frac{2M_1 (1+\theta^{-\alpha}) \mathcal{Q}}{1- e^{-\beta\theta}} + \frac{M_1 \mathcal{Q}^{1-\alpha}}{1- \alpha}, \
\Psi_2 = \frac{2M_1 (1+\theta^{-\alpha})\mathcal{Q}^{1-\alpha}}{(1- e^{-\beta\theta})(1-\alpha)}, \ \Psi_3 = M_1B(1-\alpha,1-\alpha) \mathcal{Q}^{1-\alpha}.$$
It is easy to see that if $N_1 < \min\{\Psi_1^{-1}, \Psi_3^{-1}\}$ then sequences $L'_{n}$ and $L''_{n}$ are uniformly bounded by some constants
$L'_{*}$ and $L''_{*}$ (for example, it is possible to take $L''_{*}= K_4/(1- N_1\Psi_3)$). Therefore, taking the limit in the latter chain of inequalities, we obtain that
\begin{eqnarray*}
 |u^{*}(t,y) - u^{*}(t,z)|_\alpha \le \left(L'_{*} + \frac{L''_{*}}{(t -  \tau''_m)^\alpha}\right)N_1 \|y-z\|, \quad  t \in (\tau''_m, \tau'_{m+1}], \ m \in \Z,
\end{eqnarray*}
so that 
\begin{eqnarray}
|u^*(\tilde\tau_j^1, y) - u^*(\tilde\tau_j^1, z)|_\alpha \le \left( L'_{*} + \frac{L''_{*}}{\theta^\alpha}\right)N_1\|y-z\| \quad \mbox{for} \ j=m+1 \quad \mbox{(i.e.} \ \tilde \tau_j^1 =  \tau'_{m+1}). 
\label{rizn9}
\end{eqnarray}
To complete the proof, still we need to estimate the norm $|u^*(\tilde\tau_j^1, z) - u^*(\tilde\tau_j^2, z)|_\alpha$. 
Recall that we are assuming that $\tilde \tau_j^1 \leq \tilde \tau_j^2$.  By invoking estimates (\ref{vot}) and applying  \cite[Theorem 3.5.2]{H}, we conclude that for each $\mu \in [0,1)$ there exists a positive 
$\tilde C$, which does not depend on $y \in \frak{N}$, $j \in \Z$, and such that
$$\left|\frac{d}{ds}u^*(s, z)\right|_\mu \le \tilde C (s - \tilde\tau_{j-1}^2-\theta/4)^{\alpha-\mu-1}, s \in (\tilde\tau_{j-1}^2+\theta/4,\tilde\tau_j^2).$$
Consequently, for $s \in [\tilde\tau_{j-1}^2+\theta/2,\tilde\tau_j^2),$ we have $|du^*(s, z)/ds|_\alpha \le   {4\tilde C}/{\theta}$ so that 
$$ |u^*(\tilde\tau_j^1, z) - u^*(\tilde\tau_j^2, z)|_\alpha= \left|\int_{\tilde\tau_j^1}^{\tilde\tau_j^2} \frac{d}{ds}u^*(s, z) ds \right|_\alpha
 \le  \frac{4\tilde C}{\theta} |\tilde\tau_{j}^1 - \tilde\tau_j^2| \le
 \frac{4\tilde C N_1}{\theta} \|y-z\|.
$$
The latter estimate and (\ref{pe}), (\ref{rizn9}) imply that $\|S(y) - S(z)\|\leq O(N_1)\|y - z\|$ so that 
 $S: \frak{N} \to \frak{N}$ is a contraction if $N_1$ is sufficiently small. This completes 
the proof of Theorem \ref{Thm7}.
\hfill $\Box$
\end{proof}
\begin{example}
Consider the system (\ref{ex1}), (\ref{ex2}), together with the Dirichlet boundary conditions, and assume all the conditions on coefficients of this system  imposed in Example \ref{EXA4}.  Suppose also that the subset $\{t_j\}\subset \R$ is strongly almost periodic, $\{b_j\}, \{I_j(x)\}, \{d_j(x)\}, \{\partial^k K_j(\xi,\zeta)/\partial \xi^k\},$ $ k=0,1,2,$ are scalar almost periodic (uniformly in $x$ from bounded subsets of $\R$ and $(\xi,\zeta)\in [0,l]^2$) sequences,  $a(t), b(t)$ are Bohr's almost periodic functions, 
and the mean value $\bar a$ of $a(t)$ is less than the first eigenvalue of the  operator $A$: 
$
\bar a  < {\pi^2}/{l^2}.
$ 
Then, by Theorem  \ref{Thm7}, for each given quintuple $\frak{Q} =\left\{l, a(t), \{t_j\},  \{d_j\}, \{K_j\}\right\}$ and 
each $\rho >  K(0.5)\sup_{j \in \Z}|d_j|_{H^2},$
there is a positive number $\delta$ depending only on  $\frak{Q}$ and $\rho$ such that  if $b(t) \leq \delta,\, t \in \R,$ and $|b_j|, \Lambda_j \leq \delta, \ j \in \Z$, then system (\ref{ex1}), (\ref{ex2}) has a unique positive $H^1_0(0,l)$-valued  Wexler's almost periodic solution $u(t,\cdot), \ t \in \R,$ satisfying the inequality $|u(t,\cdot)|_{H^1_0(0,l)} \leq \rho, \ t \in \R$.  Note that  the graph of $u(t,\cdot)$ intersects each surface $\Gamma_j$ of impulsive action exactly one time due to the result of Example \ref{EXA4} and the non-negativity  of the solution  $u(t,\cdot)$.  Now, $u(t,x)\geq 0$ for all $t\in \R, \ x \in (0,l),$  due to the following two facts: 1) if $
\bar a  < {\pi^2}/{l^2}$  then the Green function $G(t,\tau)$ of the equation $u_t=u_{xx}+a(t)(1-\rho b(t))u$  has the form 
\begin{eqnarray*}
G(t,s) =
 \left\{\begin{array}{l}
         e^{-A(t-s)}e^{\int_s^t(a(s)(1-\rho b(s))ds}, \ t > s, \quad A = - \frac{\partial^2}{\partial \xi^2},\\
            0, \ t \leq s, 
          \end{array} \right.
\end{eqnarray*}
and it is non-negative since  clearly $e^{-At}\geq 0$; 2) the solution $u(t,y)$ of the integral equation (\ref{nne}) is non-negative because the sequence $u_n(t,y)$ generated by (\ref{ipr})  is non-negative for all $n\in \N$ (observe that $g_j(u)(\xi)= \int_0^lK_j(\xi,\zeta)I_j(u(t,\zeta))d\zeta+ d_j(\xi)\geq 0$ and $f(t,u)= a(t)b(t)(\rho-u)u \geq 0$ for $u\in [0,\rho]$). 
It is easy to see that  actually $u(t,x) \equiv 0$ if all $d_j(\xi)\equiv 0$ and that 
$u(t,x) > 0$ for all $t\in \R, \ x \in (0,l),$ if $d_j(\xi_0) >0$ at least for one $j$ and some $\xi_0 \in (0,l) $.

\end{example}

\section*{Acknowledgments}

\vspace{-1.5mm}

\noindent  
The authors express their gratitude to
the anonymous referee  for careful reading of the manuscript and valuable comments.  
This research was supported by the CONICYT grant 80130046 (S. T. and R. H.), the Academy of Sciences of the Czech Republic grant RVO 67985840 (R. H.),  and the FONDECYT  projects 1080034, 1120709 (M. P. and V. T.).


\vspace{3.5mm}


\begin{thebibliography}{20}



\bibitem{ABET}   M.U. Akhmet, M. Beklioglu, T. Ergenc, V.I. Tkachenko, An impulsive ratio-dependent predatorÐprey system with diffusion, 
Nonlinear Analysis: Real World Applications,  {7} (2006) 1255--1267.


\bibitem{AP}  M.U. Akhmetov, N.A. Perestyuk, Periodic and almost periodic solutions of strongly nonlinear impulse systems,
{J. Appl. Math. Mech.,} {56} (1992)  829--837.

\bibitem{BJ1} E.M. Bonotto, M.Z. Jimenez,   Weak almost periodic motions, minimality and stability in impulsive semidynamical systems, J. Differential Equations 256 (2014) 1683--1701.

\bibitem{BJ2} E.M. Bonotto, M.C. Bortolan, A.N. Carvalho, R. Czaja,  Global attractors for impulsive dynamical systems - a precompact approach, J. Differential Equations 259 (2015) 2602--2625.


\bibitem{BJ3} E.M. Bonotto, M.Z. Jimenez, On impulsive semidynamical systems: minimal, recurrent and almost periodic mo- tions, Topol. Methods Nonlinear Anal. 44 (2014) 121--141.


\bibitem{PG1}  F. C\'ordova-Lepe, G. Robledo, M. Pinto, E. Gonz\'alez-Olivares, 
Modeling pulse infectious events irrupting into a controlled context: a SIS disease with almost periodic parameters, 
Appl. Math. Model., 36 (2012) 1323--1337. 

%

\bibitem{HW} A. Halanay, D. Wexler, Qualitative theory of systems with impulses, Editura Acad. Rep. Soc. Romania, Bucuresti, 1968 (in Romanian).

\bibitem{HAR} H.R. Henriquez, B. De Andrade, M. Rabelo, Existence of almost periodic solutions for a class of abstract impulsive differential equations, {ISRN Math. Anal.,} Art. ID 632687 (2011)  21 p.

\bibitem{H} D. Henry, Geometric theory of semilinear parabolic equations,
{Lect. Notes Math.,} 840, Springer, 1981.

\bibitem{EH1}  E. Hern\'andez, D.  O'Regan, 
On a new class of abstract impulsive differential equations,  
Proc. Amer. Math. Soc. 141 (2013) 1641--1649. 

\bibitem{EH2}  E. Hern\'andez, D.  O'Regan, 
Existence results for a class of abstract impulsive differential equations, Bull. Aust. Math. Soc. 87 (2013) 366--385.


\bibitem{HPh}  E. Hille, R. Phillips, Functional analysis and semi-groups, American Mathematical Society, 
Colloquium Publications,  Volume 31, 1957. 

\bibitem{LBS}  V. Lakshmikantham, D.D.  Bainov, P.S.  Simeonov, 
Theory of impulsive differential equations, World Scientific Publishing Co.,  NJ, 1989. 

\bibitem{LMP} C. Lizama,  J. Mesquita, R. Ponce, A connection between almost periodic functions defined on timescales and $\R$,  Applicable Analysis, 93 (2014) 2547--2558.

%
%


\bibitem{Pankov}  A. Pankov, Bounded and almost periodic solutions of nonlinear operator differential equations, Kluwer,  1990.
  

\bibitem{PR} M. Pinto, G. Robledo,  Existence and stability of almost periodic solutions in impulsive neural network models, {Appl. Math. Comput.,} {217} (2010) 4167--4177.

\bibitem{PG2}  G. Robledo, M.  Pinto, F. C\'ordova-Lepe, Stability of disease free state in SIS models. Almost periodic rates and pulse vaccination, BIOMAT 2011, 253--263, World Sci. Publ., Hackensack, NJ, 2012.

\bibitem{RogovY}  Y. Rogovchenko, 
Nonlinear impulse evolution systems and applications to population models,
J. Math. Anal. Appl. 207 (1997) 300--315. 

\bibitem{RT}  Yu.V. Rogovchenko, S.I.Trofimchuk, Periodic solutions of weakly nonlinear partial differential equations of parabolic type with impulse action and their stability,    Akad. Nauk Ukrain. SSR,  preprints series of  Institute of Math., 65 (1986)  44 p. 

\bibitem{Ronto}  M. Ronto,  S.I.  Trofimchuk, 
Periodic solutions of abstract impulsive systems
with non-fixed moments of impulsive effect, Publ. of University of Miskolc,
Ser. D, Natural Sciences, 36 (1995) 91-99. 

\bibitem{SP}
A.M. Samoilenko, N.A. Perestyuk, Impulsive differential equations, World Scientific, Singapore, 1995.

\bibitem{SPT} A.M. Samoilenko, N.A. Perestyuk, S.I. Trofimchuk, Generalized solutions of impulse systems and the phenomenon of pulsations, Ukrainian Math. J. 43 (1991)  610--615.

\bibitem{STr} A.M. Samoilenko, S.I. Trofimchuk, Almost periodic impulsive systems, {Differ. Eqns., } { 29} (1993) 684--691.

\bibitem{STsets} A.M. Samoilenko, S.I. Trofimchuk,
{Spaces of piecewise-continuous almost-periodic 
functions and almost-periodic sets on the line. I.} 
Ukrainian Math. J., {43}  (1991) 1501-1506. 


\bibitem{Sch} L. Schwartz, Th\'eorie des distributions,  1-2, Hermann, 1966. 

\bibitem{Sla} A. Slav\'ik, Well-posedness results for abstract generalized differential equations and measure functional differential equations. J. Differential Equations 259 (2015) 666--707.

\bibitem{S} G.T. Stamov,
Almost periodic solutions of impulsive differential equations,
{ Lect. Notes Math.,}  Springer, 2012.


\bibitem{SA} G.T. Stamov, J.O. Alzabut, Almost periodic solutions for abstract impulsive differential equations, {Nonlinear Anal.,} {72} (2010) 2457--2464.


\bibitem{OO} O.O. Struk, V.I. Tkachenko,
On impulsive Lotka--Volterra systems with diffusion, Ukrainian Math. J.,  54 (2002) 629--646. 

\bibitem{VIT} V.I. Tkachenko, On the exponential dichotomy of pulse evolution systems. Ukrainian Math. J., 46 (1994) 441--448. 

\bibitem{Tk2} V. Tkachenko, Almost periodic solutions of parabolic type equations with impulsive action, { Functional Differential Equations,} {21} (2014) 155--169.

\bibitem{Tr} S.I. Trofimchuk, Almost periodic solutions of linear abstract impulse systems, {Differ. Eqns., } {31} (1996) 559--568.

\bibitem{TrSP} S.I. Trofimchuk, Periodic and almost periodic impulsive systems.  A supplementary chapter in  "A.M. Samoilenko, N.A. Perestyuk, Impulsive differential equations, World Scientific, Singapore, 1995",  pp. 361--441.

\bibitem{WA} C. Wang,  R. Agarwal, 
A further study of almost periodic time scales with some notes and applications, Abstr. Appl. Anal. 2014, Art. ID 267384, 11 pp. 

\bibitem{WY} L. Wang, M. Yu, 
Favard's theorem of piecewise continuous almost periodic functions and its application, 
J. Math. Anal. Appl. 413 (2014) 35--46. 






\end{thebibliography}
\end{document}